\documentclass[a4paper,11pt]{article}
\usepackage{amsfonts,graphics,amsmath,amsthm,amsfonts,amscd,amssymb}
\usepackage{amsmath,latexsym,multicol,mathrsfs,stmaryrd,color,mathtools,paralist}
\usepackage{epsfig,url,pxfonts,accents}
\usepackage{fullpage}
\usepackage{flafter,multirow}
\usepackage{galois,enumerate}
\usepackage{fancyhdr}
\usepackage{hyperref}
\usepackage{textcomp}
\usepackage{tocloft}
\hypersetup{colorlinks=true, linkcolor=black}

\newcommand{\ovline}[1]{%
  \vbox{%
    \hrule height 0.6pt
    \kern0.22ex
    \hbox{%
      \kern-0.05em
      \ifmmode#1\else\ensuremath{#1}\fi
      \kern0em
    }
  }
}

\makeatletter

\def\jobis#1{FF\fi
  \def\predicate{#1}%
  \edef\predicate{\expandafter\strip@prefix\meaning\predicate}%
  \edef\job{\jobname}%
  \ifx\job\predicate
}

\makeatother

\if\jobis{proposal}%
\else
\fi

\usepackage[all]{xy}
\renewcommand{\mod}{\operatorname{mod}}

\DeclareMathOperator{\Hom}{Hom}

 \numberwithin{equation}{subsection}
 \numberwithin{footnote}{subsection}

 \newtheorem{cor}[subsection]{Corollary}
 \newtheorem{lem}[subsection]{Lemma}
 \newtheorem{prop}[subsection]{Proposition}
 \newtheorem{thm}[subsection]{Theorem}

 \newtheorem{conj}[subsection]{Conjecture}

 \newtheorem{quest}[subsection]{Question}
{
    \newtheoremstyle{upright}%
        {8pt plus2pt minus4pt}%
        {8pt plus2pt minus4pt}%
        {\upshape}%
        {}%
        {\bfseries\scshape}%
        {}%
        {1em}%
        {}%
\theoremstyle{upright}

 \newtheorem{rem}[subsection]{Remark}
 \newtheorem{defn}[subsection]{Definition}
 
 \newtheorem{exam}[subsection]{Example}

}

\newcommand{\x}{\mathscr}
\newcommand{\f}{\mathfrak}

\newcommand{\C}{\mathbb C}

\newcommand{\oo}{\mathscr O}

\newcommand{\BP}{\mathbb P}

\newcommand{\Q}{\mathbb Q}
\newcommand{\R}{\mathbb R}

\newcommand{\Z}{\mathbb Z}
\newcommand{\spec}{\mathrm{Spec}\hspace{0.5mm}}
\newcommand{\proj}{\mathrm{Proj}\hspace{0.5mm}}

\newcommand{\aut}{\mathrm{Aut}}
\newcommand{\gl}{\mathrm{GL}}
\newcommand{\pgl}{\mathrm{PGL}}

\newcommand{\cre}{\mathrm{Cr}}

\newcommand{\bir}{\mathrm{Bir}}

\newcommand{\gal}{\mathrm{Gal}}

\newcommand{\ns}{\mathrm{NS}}

\newcommand{\gr}{\mathbb{G}}

\newcommand{\wcl}{\mathrm{Cl}}
\newcommand{\hilb}{\mathrm{Hilb}}

\begin{document}
\title{Finite $p$-groups of birational automorphisms and characterizations of rational varieties}
\author{Jinsong Xu}
\date{}
\maketitle
\begin{abstract}
	We study finite $p$-subgroups of birational automorphism groups. By virtue of boundedness theorem of Fano varieties, we prove that there exists a constant $R(n)$ such that a rationally connected variety of dimension $n$ over an algebraically closed field is rational if its birational automorphism group contains a $p$-subgroups of maximal rank for $p > R(n)$. Some related applications on Jordan property are discussed.
\end{abstract}

\tableofcontents

\section{Introduction}
The group of birational automorphisms of an algebraic variety is an important birational invariant. This group is in general quite complicated and difficult to study. For example, the group $\cre_n(k) := \bir(\BP^n_k)$, known as the Cremona group of rank $n$ over a field $k$, is a huge group that attracts a lot of interests. In their work of studying Jordan property (see Definition \ref{boundjordan}) of the Cremona group, Yu. Prokhorv and C. Shramov proved an interesting result on finite $p$-groups of birational automorphisms:

\begin{thm} \cite[Theorem 1.10]{ps1} \label{psrank}
	There exists a constant $L(n)$ such that for any rationally connected variety $X$ of dimension $n$ defined over any field $k$ of characteristic zero and for any prime $p > L(n)$, every finite $p$-subgroup of $\bir(X)$ is abelian and can be generated by at most $n$.
\end{thm}

Intuitively, the $p$-group in Theorem \ref{psrank} comes from $p$-torsion subgroup of an algebraic torus in $\bir(X)$. Naturally, we may ask whether such a $p$-group is indeed contained in some subtorus of $\bir(X)$. It turns out that we don't need to restrict ourselves on rationally connected varieties, so we propose the following question.

\begin{quest} \label{finitecontinuous}
	Let $X$ be an algebraic variety over an algebraically closed field. Assume that $\bir(X)$ contains a finite subgroup $H$ that is isomorphic to $(\Z/p\Z)^r$ for a sufficiently large prime $p$. Is it true that $X$ is birational to some variety having an action of a torus or an abelian variety?
\end{quest}

The purpose of this paper is to study this question in some special cases. Our main result is a rationality criterion of rationally connected varieties by means of finite $p$-groups of birational automorphisms.

\begin{thm} \label{rationality} (see Theorem \ref{rationalrcv})
	There exists a constant $R(n)$ (depending only on $n$) such that for any $n$-dimensional rationally connected variety $X$ over an algebraically closed field of characteristic zero, if $\bir(X)$ contains a subgroup $G$ isomorphic to $(\Z/p\Z)^n$ for some prime $p > R(n)$, then $X$ is rational.
\end{thm}

We will prove a generalized form in Section \ref{main}. In the course of proving Theorem \ref{rationality}, we answer a question of Prokhorov and Shramov \cite[Questions 1.9]{ps3}, see Remark \ref{psquest} for details. As an application, we strengthen a theorem of S. Cantat \cite{c}, which characterizes rational varieties by their birational automorphism groups. In contrast with Cantat, who considered actions of nilpotent vector groups, we are primarily interested in actions of finite groups.

\begin{thm} (see Theorem \ref{cantat})
	Let $X$ be an $n$-dimensional algebraic variety over an algebraically closed field $k$ of characteristic zero. Assume that $\bir(X)$ contains a subgroup isomorphic (as an abstract group) to the Cremona group $\mathrm{Cr}_k(n)$, then $X$ is rational.
\end{thm}

We also have an affirmative answer to Question \ref{finitecontinuous} for non-unirulded varieties.

\begin{thm} \label{nruled} (see Theorem \ref{nruled1})
	Let $X$ be a non-uniruled variety over an algebraically closed field of characteristic zero. Then there exists a constant $b(X)$ (depending on the birational class of $X$) such that $\bir(X)$ contains an element of order greater than $b(X)$ if and only if $X$ is birational to a normal projective variety having an action of an abelian variety $A$.
\end{thm}

The structure of those varieties having actions of abelian varieties are well-understood. Thus we are able to give a concise characterization of algebraic threefolds whose birational automorphism groups are not Jordan. These varieties were already studied in \cite{ps4}, where they were classified into four types. Our approach does not use arithmetic of elliptic curves over function fields.

\begin{thm} \label{threefolds} (see Theorem \ref{threefolds1})
	Let $X$ be a threefold over an algebraically closed field of characteristic zero. Then $\bir(X)$ is not Jordan if and only if $X$ is birational to $\BP^1 \times (E \times^G F)$, where $E$ is an elliptic curve, or $\BP^1 \times A$ where $A$ is an abelian surface.
\end{thm}

The above theorem motivates us to raise the following conjectural characterization of algebraic varieties whose birational automorphism groups are not Jordan. Recall that a \emph{semi-abelian variety} is an extension of an abelian variety by an algebraic torus. We say that a semi-abelian variety is \emph{non-trivial} if both the abelian variety and the torus have positive dimension.

\begin{conj}
	Let $X$ be an algebraic variety over an algebraically closed field. Then $\bir(X)$ is not Jordan if and only if $X$ is birational to some variety having a faithful action of a non-trivial semi-abelian variety.
\end{conj}

Question \ref{finitecontinuous} becomes involved when the rank of $H$ is not maximal or when the variety is of mixed type. If it were fully settled for a rationally connected variety $X$, then the continuous group should be an algebraic torus, and $X$ is birational to a product $\BP^r \times Y$ for some variety $Y$. Some special cases have been studied in \cite{g} and \cite{sv}.\\

\noindent \emph{Sketch of ideas.} Let's outline the proof of Theorem \ref{rationality} by treating the case of threefolds over the field of complex numbers.


By completion, regularization and desingularization, we may start with a nonsingular projective rationally connected variety on which $G$ acts biregularly. Running a $G$-MMP, we terminate with either a $G$-Fano threefold or a nontrivial $G$-Mori fibre space \cite{ps1}.

For a Fano threefold, BABB's boundedness theorem shows that when $p$ is large enough, $G$ is contained in a maximal torus (of dimension $3$) of the neutral component of the automorphism group. Hence this Fano threefold is toric, and rationality follows.

For a $G$-Mori fibre space $f: X \to Z$, the group $G$ decomposes as
	$$
		0 \to G' \to G \to G" \to 0,
	$$
	where $G"$ and $G'$ are $p$-groups acting faithfully on $Z$ and the generic fibre $X_\eta$ respectively. Notice that $Z$ is rational since we are working over $\C$. Everything will be done if we could show that $X_{\eta}$ is rational over the function field $\C(Z)$. At this point, we are inspired by the following rationality criteria of conics by T. Bandman and Yu. Zarhin, and del Pezzo surfaces by V. Iskovskikh and Yu. Manin.

	\begin{thm} (\cite[Corollary 4.12]{bz}) \label{bz}
		Let $k$ be a field of characteristic zero that contains all roots of unity, and $C$ be a nonsingular projective curve of genus $0$ over $k$. Assume that $C$ is not $k$-rational. Then the order of any finite order element of $\aut_k(C)$ divides $2$, and the order of any finite subgroup $G \subset \aut_k(C)$ divides $4$.
	\end{thm}

	\begin{thm} (\cite[Theorem 2.4, Lemma 2.5]{ps4}) \label{im}
		Let $S$ be a del Pezzo surface over a field $k$.

			(1) $|\aut_k(S)| \leq 696729600$ if $K_S^2 \leq 5$.

			(2) Assume that $S(k) \neq \emptyset$. Then $S$ is $k$-rational if $K_S^2 \geq 5$.
	\end{thm}

If $\dim X_\eta = 1$, the generic fibre is a nonsingular projective curve of genus $0$. It follows from Theorem \ref{bz} that $X_\eta$ is $\C(Z)$-rational whenever $p > 2$. If $\dim X_\eta = 2$, the generic fibre is a del Pezzo surface over $k$ with a $k$-rational point (cf. \cite{ghs}). In view of Theorem \ref{im}, we see that $X_\eta$ is rational over $k$ if $p > 696729600$. In any case, we deduce that $X$ is rational over $\C$.

Higher dimensional cases follow a similar strategy, provided that we have an effective rationality criterion for Fano varieties over a non-closed field that is in a similar manner as Theorem \ref{bz} and Theorem \ref{im} above. We will do this in Section \ref{lag}. Another important ingredient of our approach is boundedness of mildly singular Fano varieties of fixed dimension. This is classically known in dimension $\leq 3$ \cite{am, kmmt} and recently extended to higher dimensions by the groundbreaking work of Birkar \cite{b1, b2}.

The proof of Theorem \ref{nruled} is based on a closer study of the birational automorphism group of a quasi-minimal model, and take the advantage of actions of abelian varieties. This will be done in Section \ref{sec1}. Theorem \ref{threefolds} follows readily from Theorem \ref{nruled} and a rationality criterion of geometrically rational surfaces over nonclosed field.\\

We will work over fields of \emph{characteristic zero} throughout this paper, though some results also hold in positive characteristics. There are various constants appearing in the sequel (see the list below), some of which can be carried out explicitly. To keep our ideas clear, we will not make effort on calculating their explicit values.\\

\noindent \textbf{List of constants:}
$$
\begin{array}{rlcrl}
	A(n): & \mbox{Theorem \ref{fano} (1)} & \hspace{20mm}& L(n): & \mbox{Theorem \ref{psrank}}\\
	B(n): & \mbox{Lemma \ref{torus}} && L(X): & \mbox{Theorem \ref{prank1}}\\
	C(n): & \mbox{Theorem \ref{babb} after} && M(n): & \mbox{Theorem } \ref{mink}\\
	D(n): & \mbox{Lemma \ref{uniquemax}}&& R(n): & \mbox{Theorem \ref{rationalrcv}} \\
	E(n): & \mbox{Lemma \ref{center}}&& R(X): & \mbox{Theorem  \ref{fano} (2)}\\
	F(n): & \mbox{Theorem \ref{fixed}} && b(X): & \mbox{Lemma  \ref{component}}\\
	G(n): & \mbox{Theorem \ref{group}} && \nu(X), \delta(X): & \mbox{Definition \ref{vector}}\\
\end{array}
$$

\noindent \textbf{Acknowledgement.} This work is supported by NSFC No.11701462. The author is grateful to Professors Caucher Birkar, Michel Brion, Philippe Gille, Mikhail Zaidenberg and Doctor Yi Gu for kindly answering many questions, he would also like to thank Doctor Yifan Chen for improving the presentation.

\section{Birational automorphisms on quasi-minimal models} \label{sec1}
Given a projective variety $X$ over a field $k$, it is well known that its automorphism group $\aut(X)$ has a natural structure of a group scheme that is locally of finite type over $k$. An automorphism of $X$ over $k$ corresponds to a $k$-point of $\aut(X)$. However, the scheme structure on the birational automorphism group $\bir(X)$ does not behave well, and fails to form a group scheme in general. Inspired by the work of \cite{ps2}, we will focus on birational automorphism groups of a \emph{quasi-minimal model}, a weaker notion of minimal models introduced by Prokhorov and Shramov \cite{ps2}. To simplify our notations, \emph{we will work over a fixed algebraically closed field $k$ (of characteristic zero) in this section}. We review some facts from minimal model program, details can be found in \cite{ps2}.\\

We say that a normal projective variety $X$ has terminal singularities if the canonical divisor $K_X$ is $\Q$-Cartier, and for every resolution of singularities $\pi: Y \to X$, if we write
	$$
		K_Y = \pi^*K_X + \sum a_iE_i
	$$
	where $E_i$ are exceptional divisors, then $a_i > 0$ for all $i$.

A $\Q$-divisor $M$ on $X$ is said to be $\Q$-\emph{movable} if for some integer $d > 0$ the divisor $dM$ is integral and generates a linear system without fixed components. We say that $X$ is a \emph{quasi-minimal model} if $X$ has terminal singularities, and there exists a sequence of $\Q$-movable $\Q$-Cartier $\Q$-divisors $M_j$ whose limit in the Neron-Severi space $\ns_{\Q}(X) = \ns(X) \otimes \Q$ is $K_X$. It is easy to see that the canonical divisor of a quasi-minimal model is pseudo-effective, so that $X$ cannot be covered by rational curves. Conversely, Prokhorov and Shramov \cite[Lemma 4.4]{ps2} proved that every non-uniruled algebraic variety is birational to a quasi-minimal model. Furthermore, any birational map between two quasi-minimal models is an isomorphism in codimension $1$. It follows from Hanamura \cite{h1} that the birational automorphism group of a quasi-minimal model has a natural structure of a group scheme.

\begin{thm} \label{bir}
	Let $X$ be a quasi-minimal model. Then the birational automorphism group $\bir(X)$ of $X$ has a natural structure of a group scheme. Its neutral component $\bir_0(X)$ coincides with $\aut_0(X)$, the neutral component of the automorphism group scheme $\aut(X)$. 
\end{thm}

\begin{rem}
	(1) Given a projective variety $X$, one can always endow $\bir(X)$ a scheme structure by regarding it as a subscheme of the Hilbert scheme $\hilb(X \times X)$. However, as remarked by Hanamura \cite{h1}, the scheme structure on $\bir(X)$ does not behave well and fails to be a birational invariant in general. On the other hand, we can show that a birational map $\varphi: X \dashrightarrow Y$ between \emph{quasi-minimal models} induces an isomorphism of group schemes $\varphi_{\#}: \bir(X) \to \bir(Y)$. This is the main reason we choose to work on a quasi-minimal model.
	
	(2) Given a non-uniruled variety $X$, a quasi-minimal model of $X$ is obtained by running a minimal model program (MMP) on a nonsingular projective model of $X$. Hence the resulting quasi-minimal model is always $\Q$-factorial. We will take this for granted in the sequel.
\end{rem}

The connected algebraic group $\aut_0(X)$ is an \emph{abelian variety} of dimension $h^0(X, T_X)$ because $X$ is not uniruled. Notice that the number $h^0(X, T_X)$ is independent of choice of the quasi-minimal model $X$. This leads us to introduce the following notion.

\begin{defn} \label{vector}
Let $X$ be an algebraic variety (not necessarily projective), and let $f: X \dashrightarrow Z$ be the maximal rationally connected (MRC) fibration of $X$. It is known that $Z$ is not uniruled \cite{ghs}, and the generic fibre $X_\eta$ is rationally connected. Let $Y$ be a quasi-minimal model of $Z$. We define
	$$
		\delta(X) := \dim(X_{\eta}) ~~~\mbox{ and }~~~ \nu(X) := h^0(Y, T_Y),
	$$
	where $T_Y = \x{H}om(\Omega_Y, \oo_Y)$ is the tangent sheaf of $Y$.
\end{defn}

The number $\nu(X)$ is the dimension of automorphism group scheme of the quasi-minimal model $Y$. Notice that we always have $\nu(X) \leq \dim Y$. Equality holds if and only if $Y$ is an abelian variety.


\begin{lem} \label{codim}
	Let $X$ be a $\Q$-factorial normal projective variety, $L$ be an ample $\Q$-divisor on $X$, and let $f: X \dashrightarrow X$ be a birational map which is an isomorphism in codimension $1$. Assume that $f_*L \equiv L $. Then $f$ is an isomorphism.
\end{lem}
\begin{proof}
	Note that as $X$ is $\Q$-factorial, $f_*L$ is a $\Q$-Cartier $\Q$-divisor, so that it makes sense to say $f_*L \equiv L$. By assumption, there exists big open subsets $U$ and $V$ of $X$, such that $f: U \to V$ is an isomorphism. Let $\Gamma_f \subset U \times V$ be the graph of $f$, and $\Gamma$ be the Zariski closure of $\Gamma_f$ in $X \times X$. Let $p$ and $q$ be the first and second projections of $\Gamma$ onto $X$,
	$$
	\xymatrix{
	& \Gamma \ar[rd]^q\ar[ld]_p & \\
	X \ar@{-->}[rr]^f && X  
	}
	$$
	they are all projective birational morphisms. Notice that as $f$ is an isomorphism in codimension $1$, push-forward of divisors satisfies $q_* = f_*p_*$ and $p_* = f^{-1}_*q_*$.
	
	Consider the $\Q$-Cartier $\Q$-divisor $D = p^*L - q^*f_*L$ on $\Gamma$. $D$ is $q$-nef, and $-D$ is $p$-nef. The push-forward of $D$ by $p$ (resp. by $q$) equals $0$ on $X$, by negativity lemma, both $D$ and $-D$ are effective, so $D = 0$. This implies that $p^*L = q^*f_*L \equiv q^*L$. Since the divisor
	$$
	p^*L + q^*L \equiv 2p^*L \equiv 2q^*L
	$$
	is always ample on $\Gamma$, it follows that $p$ and $q$ must be isomorphisms, a fortiori, $f$ is an isomorphism. 
\end{proof}

\begin{lem} \label{comp} \cite[Lemma 2.5]{mz}
	Let $X$ be a projective variety. Then the component group $\aut(X)/\aut_0(X)$ has bounded finite subgroups.
\end{lem}

\begin{lem} \label{component}
	Let $X$ be a ($\Q$-factorial) quasi-minimal model. Then the component group $\bir(X)/\bir_0(X)$ has bounded finite subgroups, i.e., there exists a constant $b(X)$ such that $|G_1| \leq b(X)$ for every finite subgroup $G_1 \subset \bir(X)/\bir_0(X)$.
\end{lem}
\begin{proof}
	Let $G_1$ be a finite subgroup of $\bir(X)/\bir_0(X)$. By a Theorem of Brion \cite{br3}, we may find a finite subgroup $G$ of $\bir(X)$ maps onto $G_1$, so that $G_1 \simeq G/(G \cap \bir_0(X))$. Thus it is enough to bound the order of $G/(G \cap \bir_0(X))$.
	
	By Theorem \ref{bir}, we have $\bir_0(X) = \aut_0(X)$. Fix an ample divisor $L$ on $X$.  The group $G$ acts on the Weil divisor class group $\wcl(X)$, and hence on the finite generated abelian group $V := \wcl(X)/\wcl_0(X)$ (cf. \cite[\S5]{ps2}). Let $G'$ (resp. $G"$) be the kernel (resp. the image) of the homomorphism $G \to \gl(V)$, then $G'$ acts trivially on $V$. In particular, it preserves the algebraic class, hence the numerical class of $L$. By Lemma \ref{codim}, $G'$ is contained in the automorphism group $\aut(X)$. It follows from Lemma \ref{comp} that there exists a constant $\ell(X)$ such that
	$$
		|G'/(G' \cap \aut_0(X))| \leq \ell(X).
	$$
	On the other hand, $|G"|$ is bounded by a constant say, $\gamma(X)$ (cf. Corollary \ref{finitez}). Thus
	$$
		|G/(G \cap \bir_0(X))| \leq |G"||G'/G' \cap \bir_0(X)| \leq \gamma(X)\ell(X).
	$$
	Therefore $b(X) = \gamma(X)\ell(X)$ is a desired upper bound.
\end{proof}

For varieties having actions of abelian varieties, we have a satisfactory structure theorem due to Nishi, Matsumura, Rosenlicht and Brion.


\begin{thm} \label{nmbr}
	Let $A$ be an abelian variety, acting faithfully on a normal projective variety $X$. Then $X$ is isomorphic to $A \times^H F$, where $H \subset A$ is a finite subgroup of translations and $F \subset X$ is some normal projective variety stable under $H$.
\end{thm}

The symbol $A \times^H F$ means the quotient of $A \times F$ by the diagonal action of $H$. Nontrivial examples of such varieties are given by biellipitic surfaces, which are minimal surfaces of Kodaira dimension $0$, with irregularity $q = 1$ and geometric genus $p_g = 0$. One can show that the variety $F \subset X$ in the theorem is a fibre of the homogeneous fibration on $X$ induced from the action of $A$. We refer to \cite{x} for details.

\begin{defn}
	Let $X$ be a non-uniruled variety. In the birational class of $X$, we choose a fixed quasi-minimal model $Y$, and a fixed ample divisor $L$ on $Y$. The pair $(Y, L)$ is referred as the \emph{polarized quasi-minimal model} of $X$. By our choice, it depends only on the birational class of $X$. 
\end{defn}

Our first application is to characterize non-uniruled varieties whose birational automorphism groups do not have bounded finite subgroups (see Definition \ref{boundjordan}).

\begin{thm} \label{nruled1}
	Let $X$ be a non-uniruled variety, and $b(X)$ be the constant introduced in Lemma \ref{component}. Then $\bir(X)$ contains a finite subgroup $G$ of order greater than $b(X)$ if and only if $\nu(X) > 0$. Moreover, such a variety $X$ is birational to a normal projective variety $Y$ having an action of an abelian variety $A$ of dimension $\nu(X)$, and $Y \simeq A \times^H F$ for some non-uniruled projective variety $F$.
	\end{thm}
\begin{proof}
	The question is birational in nature, so we may replace $X$ by its polarized quasi-minimal model. Lemma \ref{component} shows that if $|G| > b(X)$ for some finite subgroup $G \subset \bir(X)$, then $G$ has nontrivial intersection with the neutral component $\aut_0(X)$. Hence $\nu(X) > 0$.
	
	Conversely, if $\nu(X) > 0$, $\aut_0(X)$ is an abelian variety of dimension $\nu(X)$, which contains elements of order $d$ for every $d > 0$. This shows that $\bir(X)$ does not have bounded finite subgroups. The rest of statements follows from Theorem \ref{nmbr} with $A = \aut_0(X)$.
\end{proof}

A non-uniruled variety having vanishing irregularity $q = h^1(X, \oo_X)$ satisfies $\nu(X) = 0$. Thus Theorem \ref{nruled1} strengthens Theorem 1.8 (i) in \cite{ps2}.\\

Next, we present a characterization algebraic threefolds whose birational automorphism groups are not Jordan. Recall that the \emph{rank} of a finite group $G$ is
	$$
		r(G) := \min\{|S|: S \subset G, S \mbox{ generates } G\},
	$$
	where $|S|$ means the cardinality of $S$. For example, an \emph{elementary abelian group} $G = (\Z/p\Z)^r$ has rank $r(G) = r$.

\begin{defn} \label{boundjordan}
	Let $G$ be a group. We say that
	
	(1) $G$ has \emph{bounded finite subgroups} if there exists a constant $B$ such that $|\Gamma| \leq B$ for every finite subgroup $\Gamma \subset G$;
	
	(2) $G$ is \emph{Jordan} if there exists a constant $J$ such that for every finite subgroup $\Gamma \subset G$, we may find a normal abelian subgroup $A \subset \Gamma$ of index $[\Gamma:A] \leq J$;
	
	(3) $G$ has \emph{finite subgroups of bounded rank} (cf. \cite[Definition 2.5]{ps2}) if there exists a constant $R$ such that $r(\Gamma) \leq R$ for every finite subgroup $\Gamma$ of $G$.
\end{defn}

Given an exact sequence of groups $1 \to G' \to G \to G" \to 1$, one can show that $G$ is Jordan if either $G'$ has bounded finite subgroups and $G"$ is Jordan and has finite subgroups of bounded rank, or $G'$ is Jordan and $G"$ has bounded finite subgroups. Therefore, given that $G'$ and $G"$ are Jordan, and $G"$ has finite subgroups of bounded rank, then $G$ is not Jordan implies that both $G'$ and $G"$ do not have bounded finite subgroups. The main results of \cite{ps1} and \cite{ps2} tell us that $\bir(X)$ is Jordan if $X$ is either rationally connected or non-uniruled. Moreover, we also know that $\bir(X)$ has finite subgroups of bounded rank if $X$ is non-uniruled.

\begin{thm} \label{threefolds1}
	Let $X$ be a threefold. The group $\bir(X)$ is not Jordan if and only if $X$ is birational to $\BP^1 \times (E \times^H F)$ where $E$ is an elliptic curve, or $\BP^1 \times A$ where $A$ is an abelian surface.
\end{thm} 
\begin{proof}
	The birational automorphism groups of all threefolds listed above are not Jordan. Conversely, assume that $G := \bir(X)$ is not Jordan. Let $f: X \dashrightarrow Z$ be the MRC fibration of $X$. We may assume $Z$ is the polarized quasi-minimal model, $X$ is nonsingular projective, and $f$ is a projective morphism. By functoriality of MRC fibration, $G$ decomposes as
	$$
	1 \to G' \to G \to G" \to 1,
	$$
	where $G'$ (resp. $G"$) acts faithfully on the generic fibre $X_\eta$ (resp. on the base $Z$) as birational automorphisms. As remarked before, we know that $G' \subset \bir(X_\eta)$ and $G" \subset \bir(Z)$ are Jordan, and $G"$ has finite subgroups of bounded rank. Hence both $G'$ and $G"$ do not have bounded finite subgroups. This implies $\dim X_\eta > 0$ and $\dim Z > 0$.
	
	By Theorem \ref{nruled1}, $\aut_0(Z)$ is a nontrivial abelian variety, and $Z$ is isomorphic to $A \times^H F$. If $\dim Z = 1$, it must be an elliptic curve. A well-known theorem of Graber, Harris, and Starr \cite{ghs} says that the generic fibre $X_\eta$ has a $k(Z)$-rational point. By Theorem \cite[Theorem 1.6]{ps4}, $X_\eta$ is $k(Z)$-rational. Hence $X$ is birational to $\BP^2 \times Z$. Up to birational equivalence, we may also write $X$ as $\BP^1 \times Z \times \BP^1$, where $Z$ is an elliptic curve. If $\dim Z = 2$, then $Z$ is isomorphic to either an abelian surface $A$ or $E \times^H F$ for some nonsingular projective curve $F$. By Theorem \ref{bz}, the generic fibre $X_\eta$ is $k(Z)$-rational. Hence $X$ is birational to $\BP^1 \times Z$.
	
	Summarizing, we see that $X$ is either birational to $\BP^1 \times A$ for some abelian surface $A$, or $\BP^1 \times (E \times^H F)$ for some elliptic curve $E$ and a nonsingular curve $F$ (any genus is possible). As desired.
\end{proof}

Finally, we can bound the rank of finite groups of automorphisms for an arbitrary variety. It is easy to see that on a non-uniruled variety, the rank of a finite $p$-subgroup of $\bir(X)$ can be greater than $n$ even when $p$ is very large. For example, $(\Z/p\Z)^{2n}$ can be realized as an additive subgroup of an abelian variety of dimension $n$. We show that in fact this is almost the best upper bound.

\begin{prop} \label{prank1}
Let $X$ be an algebraic variety of dimension $n$.

(1) Assume that $X$ is not uniruled. Then there exists a constant $L(X)$ (depending on the birational class of $X$) such that for any prime $p > L(X)$, if $G \subset \bir(X)$ is a finite $p$-subgroup, then $G$ is abelian and $r(G) \leq 2\nu(X)$.

(2) In general, there exists a constant $L(X)$ (depending on the birational class of $X$) such that for any prime $p > L(X)$, if $G \subset \bir(X)$ is finite $p$-subgroup, then $G$ is an extension of two abelian $p$-groups and $r(G) \leq \delta(X) + 2\nu(X)$.
\end{prop}
\begin{proof}
	(1) As before, we may assume $X$ is the selected polarized quasi-minimal model. Let $L(X)$ be the constant $b(X)$ in Lemma \ref{component}. Given a finite $p$-group $G \subset \bir(X)$, by Lemma \ref{component}, it is contained in the neutral component $\bir_0(X)$, an abelian variety of dimension $v(X) = h^0(X, T_X)$. It follows that $G$ is abelian and of rank at most $2v(X)$.
	
	(2) Let $X \dashrightarrow Z$ the MRC fibration of $X$. The group $G$ descends to an action on $Y$, so that we have an exact sequence
	$$
		1 \to G' \to G \to G" \to 1.
	$$
	Where $G'$ (resp. $G"$) acts on the generic fibre $X_\eta$ (resp. the base $Z$) faithfully. Notice that both $G'$ and $G"$ are $p$-groups. When $p > L(X) := \max\{L(Z), L(\delta(X))\}$, $G"$ is abelian of rank $\leq 2v(X)$ by (1), and $G'$ is abelian of rank $\leq \delta(X)$ by Theorem \ref{psrank}. This gives $r(G) \leq r(G') + r(G") = \delta(X) + 2\nu(X)$.
\end{proof}

\begin{rem}
	The upper bounds in Theorem \ref{prank1} are optimal: the product of an $n$-dimensional projective space with an $m$-dimensional abelian variety admits a faithful action of $(\Z/p\Z)^{n + 2m}$ for arbitrarily large prime $p$. Notice that by Zarhin's example \cite{z}, the group $G$ in Theorem \ref{prank1} (2) is not necessarily abelian.
\end{rem}

We can also bound the rank of an arbitrary finite group of birational automorphisms of a variety by some constant depending only on the underlying variety. Modulo BABB's boundedness theorem, this is already known \cite[Remark 6.9]{ps2}. The proof is left to the reader.

\begin{prop}  \label{finiterank}
Let $X$ be an algebraic variety of dimension $n$.

	(1) There exists an constant $L_1(X) > 0$ (depending on the birational class of $X$) such that $r(G) \leq L_1(X)$ every finite subgroup $G \subset \bir(X)$;

	(2) There exists a constant $L_1(n)$ (depending only on $n$), such that $r(G) \leq L_1(n)$ for every rationally connected variety $X$ of dimension $n$ over any field and every finite subgroup $G \subset \bir(X)$;
\end{prop}

\section{Split algebraic tori} \label{lag}
In this section, $k$ is a field of characteristic zero, and $k_a$ is a fixed algebraic closure of $k$. When working with an algebraically closed field, we usually identify an algebraic variety with its rational points.

We start with a well-known theorem of Minkowski, and refer to \cite{s1} for an excellent exposition of the topic. Subsequently, we will prove some \emph{Jordan type} results on linear algebraic groups.

Given a prime number $p$ and an integer $x$, we denote by $v_p(x)$ the largest nonnegative integer such that $p^{v_p(x)} | x$. If $x$ is a real number, its integral part is denoted by $[x]$.

\begin{thm} \label{mink}
	Let $n \geq 1$ be an integer, and let $p$ be a prime number. Define
	$$
		M(n) = \prod\limits_p p^{M(n, p)}, ~~~\mbox{ where } ~~~ M(n, p) = \left[\cfrac{n}{p - 1} \right] + \left[\cfrac{n}{p(p - 1)} \right] + \left[\cfrac{n}{p^2(p - 1)} \right] + \cdots.
	$$
	Then (1) if $G$ is a finite subgroup of $\gl(n, \Z)$, we have $v_p(|G|) \leq M(n, p)$;
		
		(2) $|G|$ divides any $M(n)$ for every finite subgroup $G$ of $\gl(n,\Z)$, and $M(n)$ is the smallest number having this property.
\end{thm}

\begin{cor} (\cite[Corollary 2.14 ]{ps2}) \label{finitez}
	Let $N$ be a finitely generated abelian group. Then the group $\aut(N)$ has bounded finite subgroups.
\end{cor}

An algebraic group $D$ over $k$ is \emph{diagonalizable} if $D \otimes_k k_a$ is isomorphic to a closed subgroup of $\gr_{m, k_a}^r$. A connected diagonalizable group is called an algebraic torus. One can show that an algebraic torus is nothing but a connected algebraic group $T$ over $k$ such that $T \otimes_k k_a \simeq \gr_{m, k_a}^r$. A diagonalizable group $D$ is called \emph{split over} $k$ if $D$ is isomorphic to a closed subgroup of $\gr_{m, k}^r$ over $k$.

Given a diagonalizable group $D$, denote its character group by $X^*(D) := \Hom(D_{k^a},\gr_{m, k_a})$ (written additively). This is a finitely generated abelian group. The absolute Galois group $\gal(k_a/k)$ acts continuously on $X^*(D)$ via
	$$
		(\prescript{\sigma}{}{\chi})(s) = \sigma(\chi(\sigma^{-1}(s))), \hspace{5mm} \chi \in X^*(D), ~\sigma \in \gal(k_a/k), ~s \in D,
	$$
	where $\gal(k_a/k)$ and $X^*(D)$ are endowed respectively, with the Krull topology and the discrete topology. The image of the homomorphism $\gal(k_a/k) \to \aut(D)$ is always a finite group.

The assignment $T \rightsquigarrow X^*(T)$ defines a (contravariant) functor $X^*$ from the category of diagonalizable groups over $k$ to the category of finitely generated abelian groups with a continuous action of $\gal(k_a/k)$. One knows that $X^*$ is an anti-equivalence of categories and takes exact sequences to exact sequences \cite[Theorem 12.23]{m}\footnote{Diagonalizable groups (resp. groups of multiplicative type) in the book \cite{m} correspond to split diagonalizable groups (resp. diagonalizable groups) in our setting.}. Under this equivalence, split diagonalizable groups over $k$ correspond to finitely generated abelian groups on which $\gal(k_a/k)$ acts trivially, and vice verse.

A closed subgroup of a split diagonalizable group is always split. Conversely, we are interested in knowing when a split subgroup of a torus is contained in a split subtorus. A simple consideration shows that this is not always the case: the $1$-dimensional real torus $S = \spec \R[x, y]/(x^2 + y^2 - 1)$ contains a split subgroup $D = \spec \R[x]/(x^2 - 1)$, but $S$ itself is not split over $\R$.
  
\begin{lem} \label{splittorus}
	Let $T$ be an $n$-dimensional algebraic torus over $k$, and let $D \subset T$ be a closed finite subgroup scheme. If $D$ is split over $k$ and $d = |D(k_a)|$ is coprime to $M(n)$, then $D$ is contained in a split subtorus $S \subset T$.
\end{lem}
\begin{proof}
	Applying the functor $X^*$ to the exact sequence of diagonalizable $k$-groups
	$$
	1 \to D \to T \to T/D \to 1,
	$$
	we get an exact sequence of Galois modules
	$$
	0 \to X^*(T/D) \to X^*(T) \xrightarrow{|_D} X^*(D) \to 0. \eqno{(1)}
	$$
	Let $G$ be the image of the homomorphism $\gal(k_a/k) \to \aut(T) \simeq \gl(n, \Z)$, and keep in mind that $g := |G|$ divides $M(n)$. Since $D$ is split and finite over $k$, the group of Galois invariant characters $X^*(D)^G$ coincides with $X^*(D)$, which is (non-canonically) isomorphic, as an abstract group, to $D(k_a)$.
	
	Define the map $\pi: X^*(T) \to X^*(T)^G$ by $\pi(\chi) = \sum_{\sigma \in G} \prescript{\sigma}{}{\chi}$. It is a homomorphism of $G$-modules.
		
	Since $d$ and $g$ are coprime, we can choose a positive integer $m$ such that $mg \equiv 1 (\mod d)$. One verifies that $m\pi(\chi)|_D = m(\sum\prescript{\sigma}{}{\chi})|_D = mg\chi|_D = \chi|_D$, i.e., the following diagram commutes:
	$$
	\xymatrix{
	X^*(T) \ar[d]^\pi \ar[rr]^{|_D} & & X^*(D) \ar@{=}[d]\\
	X^*(T)^G \ar[r]^{|_D} & X^*(D)^G \ar[r]^-{\times m} & X^*(D)^G
	}
	$$
	Let $N$ be the image of $\pi$ in $X^*(T)^G$. Then the restriction $X^*(T) \to X^*(D)$ factors into surjections $X^*(T) \to N \to X^*(D)$ of Galois modules. Notice that $N$ is a free abelian group on which $G$ acts trivially, it corresponds to a split subtorus $S \subset T$. Thus we get the desired closed embeddings $D \hookrightarrow S \hookrightarrow T$.
	\end{proof}

\begin{cor} \label{split1}
	Let $p > n + 1$ be a prime number, $k$ be a field that contains a primitive $p$-th root of unity, and $T$ be an $n$-dimensional torus over $k$. If $T(k)$ has a subgroup $D$ isomorphic to $(\Z/p\Z)^r$ for some $r > 0$, then $D$ is contained in a split subtorus of $T$ of dimension $\geq r$.
\end{cor}
\begin{proof}
	The group $D$ can be viewed as a constant group scheme over $k$. It splits over $k$ since $k$ contains all $p$-th roots of unity. The prime number $p$ is coprime to $M(n)$ because by Theorem \ref{mink} we have $M(n, p) = 0$ when $p > n + 1$. Hence $D$ is contained in a split subtorus $S$ of $T$ by Lemma \ref{splittorus}. Clearly the dimension of $S$ is at least $r$.
\end{proof}

\begin{exam} \label{realcircle} 
	The assumption that $k$ contains a primitive root of unity is necessary because of the following example. Let $k = \R$ be the field of real numbers, and $S = \spec \R[x, y]/(x^2 + y^2 - 1)$. $S$ is a $1$-dimensional real algebraic torus, whose $\R$-points form the circle group
	$$
		S(\R) = \{(x, y) \in \R^2| x^2 + y^2 = 1\}.
	$$
	Hence $C(\R)$ contains elements of arbitrarily large order, but it is not isomorphic to $\gr_{m}$ over $\R$.
\end{exam}

	


Recall that a \emph{maximal torus} $T$ of a linear algebraic group $G$ is a subtorus of $G$ which is not contained in another one. The \emph{rank} $r(G)$ is the dimension of a maximal torus of $G$. As maximal tori of a linear algebraic group are conjugate to each other, the rank is well defined. Beware that the rank of an algebraic group is not the same thing as that of a finite group, but this will be clear from the context when they appear. We also remind the reader that a maximal torus of an algebraic group $G$ over $k$ is not necessarily defined over the same field $k$.

\begin{thm} \cite[Theorem 4.2]{ps1} \label{fixed}
	There exists a constant $F(n)$ (depending only of $n$) such that for any rationally connected variety $X$ of dimension $n$ over an algebraically closed field of characteristic zero, and for any finite subgroup $G \subset \aut(X)$, there exists a subgroup $N \subset G$ of index at most $F(n)$ acting on $X$ with a fixed point.
\end{thm}

\begin{lem} \label{torus} 
	There exists a constant $B(n) > 0$ such that for any connected linear algebraic group $G$ of rank $\leq n$ over any algebraically closed field $k$, and any finite subgroup $H \subset G$, there exists a subgroup $N \subset H$ of index $[H:N] \leq B(n)$ such that $N$ contained in a maximal torus of $G$.
	
	In particular, if $H \subset G$ is a finite $p$-subgroup, then $H$ is contained in a maximal torus whenever $p > B(n)$.
\end{lem}
\begin{proof}
	Take a Borel subgroup $B$ of $G$, the homogeneous space $X = G/B$ is a projective rationally connected variety, whose dimension is bounded by a constant $\beta(n)$ depending only on $n$. The group $H$ acts on $X$ by translation. By Theorem \ref{fixed}, there exists a subgroup $N \subset H$ of index at most $F(\beta(n))$ acting on $X$ with a fixed point, say $xB$, so that $gxB = xB$ for all $g \in N$. Thus $N$ is contained in the Borel subgroup $xBx^{-1}$ of $G$. Note that $N$ consists of semi-simple elements. A subgroup of semi-simple elements of a connected solvable is always contained in a maximal torus. Hence $N$ is contained in a maximal torus $T$ of $xBx^{-1}$. $T$ is clearly a maximal torus of $G$ as well. The constant $B(n)$ can be taken to be $F(\beta(n))$.
\end{proof}

\begin{rem}
Since the normalizer of a Borel subgroup is always itself, Lemma \ref{torus} is essentially equivalent to find the fixed point of action of $H$ on $G/B$. In the above argument in order to obtain the constant $B(n)$, we use the almost fixed point property of actions of finite groups on rationally connected varieties, which depends on the boundedness theorem of Fano varieties.

When $G$ is semisimple and assume for simplicity $H \subset G$ is a $p$-group, the author learned from Brion that Lemma \ref{torus} follows from a theorem of Springer and Steinberg, which says that $H$ normalizes a maximal torus $T$ of $G$. If $p$ does not divide the order of the Weyl group $W = N_G(T)/T$ (whose size can be bounded by the rank of $G$), then the image of $H$ in $W$ is trivial, hence $H$ is contained in $T$. As this is a question of algebraic groups, we believe a proof of Lemma \ref{torus} using only linear algebraic groups is possible, and finding explicit value of $B(n)$ might be interesting.

Lemma \ref{torus} fails if the prime $p$ is too small: the group $\pgl(2, \C)$ contains a subgroup $H \simeq \Z/2\Z \times \Z/2\Z$, but there is no maximal torus or Borel subgroup containing $H$. 
\end{rem}

Given a connected semisimple algebraic group $G$ of rank $n$ over an algebraically closed field, it is known that the order of its center is bounded by a constant depending only on $n$ (for example, one may take $n^n$ by \cite[Lemma 2.9]{mz}). Observe that the center is the intersection of all maximal tori in $G$. In what follows, we will need a more general result of this type.

\begin{lem} \label{center}
	There exists a constant $E(n)$ such that for any connected reductive group $G$ of rank $\leq n$ over any algebraically closed field $k$, and any nonempty set $\x{T}$ of maximal tori of $G$, the intersection
	$$
		D =  \bigcap_{T \in \x{T}} T
	$$
	has at most $E(n)$ connected components.
\end{lem}
\begin{proof}
	Observe that $D$ contains the radical $R(G)$ of $G$ since every maximal torus of $G$ contains $R(G)$. Passing to the quotient $G/R(G)$ if necessary, we may assume $G$ is semi-simple.
	
	Fix a maximal torus $T_0 \in \x{T}$. The group $D$ is a diagonalizable subgroup of $T_0$. To bound the number of connected components of $D$, it is equivalent to bound the order of the torsion subgroup of its character group $X^*(D) \simeq X^*(T_0)/L$, where $L \subset X^*(T_0)$ is the subgroup of characters of $T_0$ vanishing on $D$.
	
	As usual, the Lie algebra $\f{g}$ of $G$ decomposes into a direct sum
	$$
		\f{g} = \f{t}_0 \oplus\bigoplus\limits_{\alpha \in \Phi} \f{g}_{\alpha}, 
	$$
	where $\f{t}_0$ is the Lie algebra of $T_0$, $\Phi$ is the set of roots with respect to the pair $(G, T_0)$, and $\f{g}_{\alpha} = \{v \in \f{g}| \mathrm{Ad}_t(v) = \alpha(t) v \mbox{ for all } t \in T_0\}$ is the root space corresponding to $\alpha \in \Phi$. Let $Q \subset G$ be the closed subgroup generated by the tori in $\x{T}$. Then $Q$ is connected \cite[Proposition 2.2.6]{sp}, its Lie algebra $\f{q}$ takes the form
	$$
		\f{q} = \f{t}_0 \oplus\bigoplus\limits_{\alpha \in \Lambda} \mathfrak{g}_{\alpha}
	$$
	for some subset $\Lambda \subset \Phi$. 
	
	\emph{Claim: $\Lambda \subset L$ and $L/\Z\Lambda$ is torsion.}

	\emph{Proof:} Given any $x$ in the center $Z(Q)$ of $Q$, it centralizes each torus $T \in \x{T}$, thus $x$ lies in $Z_G(T) = T$ because $G$ is semi-simple. This shows that $Z(Q) \subset \bigcap T = D$. Conversely, any element of $D$ lies in the center $Z(Q)$ since the tori in $\x{T}$ generates $Q$. We deduce that $D = Z(Q)$. Consider the adjoint representation $\mathrm{Ad}: T_0 \to \gl(\f{q})$ on the Lie algebra $\f{q}$. Pick $t \in T_0$, and $v  = v_0 + \sum v_\alpha \in \f{q}$, we have $\mathrm{Ad}_t(v) = v_0 + \sum \alpha(t)v_\alpha$. Since the kernel of $\mathrm{Ad}$ is the center $Z(Q)$, we see that
		$$
			t \in D ~~\Leftrightarrow~~ t \in \mathrm{Ker}(\mathrm{Ad}) ~~\Leftrightarrow~~ \alpha(t) = 1 \mbox{ for all }\alpha \in \Lambda.
		$$
		This proves $\Lambda \subset L$. The assertion $L/\Z\Lambda$ is torsion follows from the fact that roots in $\Lambda$ vanishes simultaneously only on $D$. The claim is proved.
	
	Now we have an exact sequence
	$$
		0 \to L/\Z\Lambda \to X^*(T_0)/\Z\Lambda \to X^*(T_0)/L \to 0
	$$
	The group $L/\Z\Lambda$ is torsion, hence the inverse image of the torsion subgroup of $ X^*(T_0)/L$ in $ X^*(T_0)/\Z\Lambda$ is the torsion subgroup of $X^*(T_0)/\Z\Lambda$. Now it is enough to bound the order of torsion subgroup of $X^*(T_0)/\Z\Lambda$.
	
	It is well known that up to isomorphism, there exists finitely many connected semi-simple algebraic group of rank $\leq n$, hence there exists finitely many isomorphism classes of reduced semi-simple root data of rank $\leq n$, this in turn, implies that there are only finitely many possibilities of abelian groups of the form $X^*(T_0)/\Z\Lambda$ of rank $\leq n$. Therefore, there exists a constant $E(n)$ depending only on $n$ such that
	$$
		|\mbox{torsion subgroup of } X^*(T_0)/\Z\Lambda| \leq E(n)
	$$
	for all root data $(X^*(T_0), \Phi)$ arising from a semi-simple group of rank $\leq n$ and all possible $\Lambda \subset \Phi$. This proves the lemma.
\end{proof}

\begin{exam} \label{solvable}
	The constant $E(n)$ in the above lemma depends only on the rank of the reductive group. We remark that Lemma \ref{center} is incorrect if we drop the assumption that $G$ is reductive. Consider for each integer $\ell \geq 1$ the semi-direct product $S_\ell := \gr_m \ltimes \gr_a$ over an algebraically closed field $k$, where $\gr_m$ acts on $\gr_a$ by $(\lambda, x) \mapsto \lambda^\ell x$. $S_\ell$ is a connected non-reductive solvable group of rank $1$, and the intersection of all maximal tori is the center $Z(S_\ell)$ of $S_\ell$. It is not difficult to see that
	$$
		Z(S_\ell) = \{(\lambda, 0) | \lambda^\ell = 1, \lambda \in k^*\}
	$$
	is a cyclic group of order $\ell$. Therefore the order of $Z(S_\ell)$ is unbounded for $\ell \geq 1$.
\end{exam}

\begin{lem} \label{uniquemax} 
	There exists a constant $D(n)$ satisfying the following property: for any connected reductive group $G$ of rank $n$ over a (possibly non-closed) field $k$, and any finite group $H \subset G(k)$, we may find a subgroup $N \subset H$ of index at most $D(n)$ such that $N$ is contained in a $T(k)$ for some subtorus $T$ of $G$, defined over $k$.
\end{lem}
\begin{proof}
	Let $G_{k_a} := G \times_k k_a$. We may identify $G(k)$ with a subgroup of $G_{k_a}(k_a) = G(k_a)$. We know from Proposition \ref{torus} that $H$ has a subgroup $N_1$ of index $[H:N_1] \leq B(n)$ that is contained in at least one maximal torus of $G_{k_a}$. Let $\x{T}$ be the set of all maximal subtori of $G_{k_a}$ containing $N_1$. By Lemma \ref{center}, we may find a subgroup $N \subset N_1$ of index $[N_1:N] \leq E(n)$ such that $N$ is contained in the neutral component $S_0$ of $S = \bigcap_{T \in \x{T}} T$, a nontrivial algebraic torus of rank $\geq r(N)$.
	
	The absolute Galois group $\Gamma := \gal(k_a/k)$ acts on $G(k_a)$, and fixes all elements in $N_1$. Moreover, an automorphism $g \in \Gamma$ takes a maximal torus containing $N_1$ to another maximal torus containing $N_1$. Hence $S$, as well as its neutral component $S_0$, is invariant under $\Gamma$. It follows that $S_0$ descends to $k$-subtorus $T \subset G$. We may take $D(n) = B(n)E(n)$.
\end{proof}

Summarizing the above lemmas, we obtain the following theorem which is interesting in its own right. Recall that a \emph{Levi subgroup} of a connected linear algebra group $G$ over a field $k$ is a closed $k$-subgroup $L$ of $G$ such that the quotient morphisms $G \to G/R_u(G)$ restricts to an isomorphism of algebraic groups $L \simeq G/R_u(G)$. A Levi subgroup is reductive, and has the same rank as that of $G$. In characteristic zero, it is known that Levi subgroups of a linear algebraic group always exist \cite[Chapter VIII, Theorem 4.3]{ho}.

\begin{thm} \label{group}
	Let $D(n)$ be the constant introduced in Lemma  \ref{uniquemax}, and set $G(n) = \max\{D(n), n + 1\}$. Assume that
	
	(1) $p > G(n)$ is a prime number,
	
	(2) $k$ is a field containing a primitive $p$-th root of unity,
	
	(3) $G$ is a connected linear algebraic group of rank $n$ over $k$,
	
	(4) $G(k)$ contains a subgroup $H$ isomorphic to $(\Z/p\Z)^{r}$ for some $r \geq 1$.
	
	\noindent Then $G$ contains a split $k$-subtorus $T$ of rank $\geq r$.
\end{thm}
\begin{proof}
	Let $L$ be a Levi subgroup of $G$, i.e., a closed subgroup of $G$ that projects isomorphically onto the connected reductive group $G/R_u(G)$.
	$$
	\xymatrix{
	&&L \ar@{^(->}[d]\ar[dr]^-{\simeq} &&\\
	1 \ar[r] & R_u(G) \ar[r] & G \ar[r]^-\pi & G/R_u(G) \ar[r] & 1
	}
	$$
	Since the group $H \subset G(k)$ consists of semi-simple elements, it maps isomorphically into a subgroup of $L(k)$ via $ G \to G/R_u(G) \to L$. Therefore, replace $H$ by its image in $L$, we may assume $H$ is contained in (the group of $k$-points of) a reductive subgroup $L \subset G$. Notice that the rank of $L$ is $n$.
	
	By Lemma \ref{uniquemax} applied to the group $L$, we see that $H$ is contained in a $k$-subtorus $T_1$ of $L$ since $p > D(n)$. Furthermore, from Corollary \ref{split1}, we know that $H$ is contained in a split subtorus $T$ of $T_1$ because $p > n + 1$. Clearly, $T$ has rank $\geq r$. This proves the theorem. 
\end{proof}

\begin{rem} 
	In order to get the constant $D(n)$ in the theorem, we pass to a Levi subgroup $L$ of $G$ at the cost of losing control of the original $H$. The author is not sure if one can prove that $H$ itself is contained in a split $k$-torus. See Remark \ref{psquest} for another comment.
\end{rem}

In the next section, we will need a special case of a theorem of Rosenlicht (cf. \cite{r1, r2}), which roughly says that a \emph{split torus} $T \simeq \gr_m^r$ acting on a variety indeed \emph{splits} that variety.

\begin{thm} \label{split} 
	Let $T \simeq \gr_m^r$ be a split torus over a field $k$, $X$ be a $k$-variety, and $\alpha: T \times X \to X$ be a faithful action defined over $k$. Then $X$ is birational over $k$ to $\BP^r \times Z$ for some $k$-variety $Z$.
\end{thm}
\noindent \emph{Sketch of proof.} For a faithful action of a torus on a variety, the stabilizer of a general point is trivial. Thus shrinking $X$ if needed, we may assume that the action is free. By a theorem of Rosenlicht, we may further shrink $X$, so that the geometric quotient $\pi: X \to X/G$ exists. Then $\pi$ is a $T$-torsor over $X/G$. If one can find a rational section $s: X/G \dashrightarrow X$, then $X$ is birational to $\gr_m^r \times X/G$. The existence of rational sections follows from Hilbert theorem $90$, cf. \cite{r2}. \qed\\

	Theorem \ref{split} fails if the torus $T$ is not split over $k$. For example, over the real numbers the one dimensional torus
	$$
		S = \spec \R[x_1, x_2]/(x_1^2 + x_2^2 - 1)
	$$
	is not split over $\R$. It acts faithfully on the conic
	$$
		C = \proj \R[x_1, x_2, x_3]/(x_1^2 + x_2^2 + x_3^2)
	$$
	by rotating the first two coordinates. However, $C$ is not birational to $\BP^1$ over $\R$ since it has no $\R$-points.


\section{Characterizations of rational varieties} \label{main}
The group scheme of automorphisms of a projective variety preserving the linear class of an ample divisor is an algebraic group. In particular, the automorphism group of a terminal Fano variety $X$ is an algebraic group because the anti-canonical divisor $-K_X$ is preserved by every automorphism. Moreover, since terminal Fano variety has zero irregularity, $\aut(X)$ is a linear algebraic group. Notices that $\aut(X)$ is not necessarily connected, for instance, the group
	$$
		\aut(\BP_\C^1 \times \BP_\C^1) = (\pgl(2, \C) \times \pgl(2, \C)) \rtimes (\Z/2\Z)
	$$
	has two connected components. In order to obtain a uniform bound for the number of components of $\aut(X)$, we have to employ the boundedness property of terminal Fano varieties of fixed dimension, which is a special case of the Borisov-Alexeev-Borisov-Birkar's theorem \cite{b2}.

\begin{thm} \label{babb}
	The set $\x{F}_n$ of all terminal Fano variety of dimension $n$ over complex numbers $\C$ forms a bounded family. That is, there exists a projective morphism $\Phi : \mathcal{X} \to T$ between schemes of finite type over $\C$ such that every $X \in \x{F}_n$ is isomorphic to a closed fibre of $\Phi$.
\end{thm}

By flat stratification, the number of connected components of automorphism groups scheme of a terminal Fano variety of dimension $n$ over $\C$ is uniformly bounded above by a constant $C(n)$ depending only on $n$.

Given a $n$-dimensional terminal Fano variety $X$ over a (not necessarily closed) field $k$, we may find a terminal Fano variety $X_1$ over a finitely generated subfield $k_1 \subset k$ such that $X = X_1 \times_{k_1} k$. Embeds $k_1$ into $\C$, we get a terminal Fano variety $X_{\C} := X_1 \times_{k_1} \C$ over $\C$. Since the formation of $\aut(X)$ commutes with base change, so the number of geometric connected components of $\aut(X)$ and that of $\aut(X_{\C})$ are the same, hence they are all bounded by $C(n)$.\\

We are ready to prove Theorem \ref{rationality}. First we treat a special case.

We say that a variety $X$ of dimension $n$ over a field $k$ is \emph{toric} if there exists a split $k$-torus $T \simeq \gr_{m, k}^n$ acting faithfully on $X$ over $k$. By Theorem \ref{split}, a toric variety over $k$ is $k$-rational. 

\begin{thm} \label{fano}
	(1) There exists a const $A(n)$ (depending only on $n$) such that for any prime $p > A(n)$, any field $k$ containing a $p$-th primitive root of unity, and any terminal Fano variety $X$ of dimension $n$ over $k$, if $\aut(X)$ contains a subgroup $H$ isomorphic to $(\Z/p\Z)^n$, then $X$ is toric over $k$. 
	
	(2) Let $X$ be a projective rationally connected variety of dimension $n$ over a field $k$. There exists an integer $R(X)$ such that if $\aut(X)$ contains a subgroup $H$ isomorphic to $(\Z/p\Z)^n$ for some prime $p > R(X)$, and if the field $k$ contains a primitive $p$-th root of unity, then $X$ is toric over $k$.
	
	In both cases, $X$ is $k$-rational.
\end{thm}
\begin{proof}
	(1) By Theorem \ref{babb}, the set of all terminal Fano varieties of dimension $n$ forms a bounded family. In particular, the number of geometric connected components of the automorphism group $\aut(X)$ is uniformly bounded above by a constant $C(n)$. Thus, if $p > C(n)$, then $H$ is contained in (the group of $k$-points of) the neutral component $\aut(X)_0$, a connected linear algebraic group over $k$ of rank $\leq n$. On the other hand, by Theorem \ref{group}, $\aut(X)_0$ contains a split $k$-subtorus $T$ of rank $\geq n$ if $p \geq G(n)$. This shows that $T \simeq \gr_{m, k}^n$.
	
	By Rosenlicht's Theorem \ref{split}, $X$ is toric and birational to $\BP_k^n$ over $k$. It is enough to take $A(n) = \max\{C(n), G(n)\}$.
	
	(2) By Lemma \ref{comp}, there exists a constant $\ell(X)$ such that any finite subgroup of $\aut(X)/\aut_0(X)$ has at most $\ell(X)$ elements. Thus $H$ is contained in (group of $k$-points of) the neutral component $\aut_0(X)$ if $p > \ell(X)$. The rest of proof is similar to (1). We may take $R(X) = \max\{\ell(X), G(n)\}$. 
\end{proof}

\begin{rem} \label{psquest}
	(1) If the base field $k$ is algebraically closed, Theorem \ref{fano} (2) answers a question of Prokhorov and Shramov \cite[Question 1.9]{ps3}. A more straightforward argument goes like this: by Lemma \ref{comp}, the $p$-group $H$ is contained in $\aut_0(X)$ if $p$ is very large; Lemma \ref{torus} tells us that $H$ is contained in a maximal torus (of rank $n$) of $\aut_0(X)$ if $p > B(n)$, hence $X$ is toric and rational because we are working over an algebraically closed field.
	
	(2) When $k$ is not algebraically closed, and when dealing with family of Fano varieties, as remarked in Section 2, in order to get an effective constant $A(n)$, we pass to a Levi subgroup of $\aut_0(X)$ to avoid the possibility of Example \ref{solvable}. However, in this concrete situation, it is possible to exclude Example \ref{solvable} by a further application of boundedness theorem, and it also seems feasible to prove that $H$ itself is contained in a split subtorus of $\aut_0(X)$.
	
	(3) In the above theorem, one can prove more generally that if the $r(H)< \dim X$, then $\aut(X)$ contains a split subtorus of dimension $r(H)$, so that $X$ splits birationally into a product of a rational variety with some other variety, see \cite{g} for a similar result on forms of flag varieties. 
\end{rem}

\begin{exam}
	Let's work out the simplest case by finding a value of $A(1)$ in Theorem \ref{fano} (2). It also gives a rationality criterion for conics that is similar to Theorem \ref{bz} of Bandman and Zarhin.
	
	\emph{Claim: Let $p \geq 3$ be a prime number, and $k$ be a field that contains a primitive $p$-th root of unity. Let $C$ be a geometrically rational curve over $k$, and suppose that $\varphi \in \bir(C)$ is a birational automorphism of order $p$. Then $C$ is $k$-rational. In other words, we can take $A(1) = 2$ in Theorem \ref{fano}.}

	Replace $C$ by the normalization of its completion, we may assume $C$ is a nonsingular projective curve, so that $\bir(C) = \aut(C)$. Passing to an algebraic closure $k_a$ of $k$,  we see that $\varphi$ is contained in a torus $T$ (which is necessarily maximal since $\pgl(2, k_a)$ has rank $1$). Since any two tori in $\pgl(2, k_a)$ have two pairs of distinct fixed points on $\BP^1$, $T$ is the unique torus containing $\varphi$. Thus $T$ is fixed by the absolute Galois group $\gal(k_a/k)$, and hence defined over $k$. Now by Corollary \ref{split1}, $T$ splits over $k$ because $k$ contains a primitive $p$-th root of unity. It follows from Theorem \ref{split} that $C$ is $k$-rational.
\end{exam}

We now prove the Theorem \ref{rationality}.

\begin{thm} \label{rationalrcv}
	There exists a constant $R(n)$ such that for any prime $p > R(n)$, and any rationally connected variety $X$ of dimension $n$ over a field $k$ that contains all roots of unity, if $\bir(X)$ contains a subgroup $G$ isomorphic to $(\Z/p\Z)^n$, then $X$ is $k$-rational.
\end{thm}
\begin{proof}
	We proceed by induction on $n = \dim X$. When $n = 0$, the theorem trivially holds with $R(0) = 1$.
	 
	As before, replacing $X$ by a projective completion, regularizing the action of $G$ on $X$, and replacing by a equivariant resolution of singularities, we may assume that $X$ is a non-singular projective variety over $k$ on which $G$ acts biregularly.
	
	Run a $G$-MMP on $X$, we terminate with a $G$-Mori fibre space $f: Y \to Z$ with $\dim Z < n$. The group $G$ decomposes as
	$$
		0 \to G' \to G \to G" \to 0
	$$
	where $G"$ and $G'$ are elementary abelian $p$-groups that act faithfully on the base $Z$ and the generic fibre $Y_\eta$ respectively. Thus we may identify $G"$ (resp. $G'$) as a subgroup of $\aut_k(Z)$ (resp. $\aut_{k_\eta}(Y_\eta)$).
	
	Notice that both $Z$ and $Y_\eta$ are rationally connected. We know from Theorem \ref{psrank} that $r(G') \leq \dim Y_\eta$ and $r(G") \leq \dim Z$ if $p > \max\{L(\dim Z), L(\dim Y_\eta)\}$. By our assumption $r(G) = \dim X$, we deduce that $r(G') = \dim Y_\eta$ and $r(G") = \dim Z$. In other words, $G'$ and $G"$ are always of maximal rank if $p$ is sufficiently large.
	
	By induction, $Z$ is rational over $k$ if $p > \max\{L(\dim Z), R(\dim Z)\}$, and $Y_\eta$ is rational over $k(Z)$ if $p > \max\{L(\dim Y_\eta), A(\dim Y_\eta)\}$ by Theorem \ref{fano} (2). Set
	$$
		R(n) = \max\{R(0), R(1), \cdots, R(n-1), A(1), A(2), \cdots, A(n), L(0), L(1), \cdots, L(n)\}
	$$
	It follows that $X$ is rational over $k$ whenever $p > R(n)$.
\end{proof}

\begin{rem}
	In Question \ref{finitecontinuous}, when the rank of the $p$-group $G$ is not maximal, i.e. $r(G) < \dim X$, we can still run a $G$-MMP, and terminate at a $G$-Mori fibre space. Again we will have an exact sequence $1 \to G' \to G \to G" \to 1$. But the problem is that in this case $G'$ or $G"$ might be trivial. Even if they are not trivial, by induction we may assume either the generic fibre or the base has an action of a torus, unlike the case of Theorem \ref{fano}, we do not know how to lift this continuous group action (especially from the base) to an action on the total space $X$.
\end{rem}

In \cite{c}, Cantat characterized rationality of a complex algebraic variety by means of its birational automorphism group. Applying the results developed above, we obtain a generalization of his main theorem. An interesting feature of our approach is that the base field can be non-closed.

\begin{thm} \label{cantat}
	Let $X$ be an algebraic variety of dimension $n \geq 1$ over a field $k$ that contains all roots of unity. Assume that $\bir(X)$ contains a subgroup $G$ that is isomorphic to $\mathrm{Cr}_n(k)$ (as an abstract group), then $X$ is $k$-rational.
\end{thm}
\begin{proof}
	The Cremona group $\mathrm{Cr}_n(k)$ contains the projective linear group $G = \pgl(n+1, k)$ as a subgroup. The latter group, in turn, contains subgroups isomorphic to $(\Z/p\Z)^n$ for all prime $p$ because the field $k$ contains all roots of unity. In view of Theorem \ref{rationalrcv}, the only thing we need to show is that $X$ is rationally connected.
	
	Let $f: X \dashrightarrow Z$ be the MRC fibration of $X$. Assume that $\dim Z > 0$. As before, we may assume $Z$ is a quasi-minimal model. By the functoriality of MRC fibration, the group $G$ acts on $Z$ as birational automorphisms, thus it induces a homomorphism of abstract groups $G \to \bir(Z)$. Since $\pgl(n+1, k)$ is a non-abelian simple group for $n \geq 1$ and perfect field $k$, this homomorphism is either trivial or injective. The latter case is impossible. For otherwise, we may identify $G$ as a subgroup of the birational automorhpism group $\bir(Z)$, so that $G \cap \bir_0(Z) = G$ or $\{id\}$. The first case implies that $G$ is abelian, the second case implies that $G$ embeds into the group $\bir(Z)/\bir_0(Z)$. Both of the cases are absurd.
	
	Thus $G$ acts trivially on $Z$. In particular, $G$ acts faithfully on the generic fibre $X_\eta$ of $f$. Since $G$ contains subgroups isomorphic to $(\Z/p\Z)^n$ for arbitrarily large prime $p$, by Theorem \ref{psrank}, we deduce that $\dim X_{\eta} = n$. This proves that $\dim Z = 0$, and $X$ is therefore rationally connected.
	\end{proof}


\noindent Department of Mathematical Sciences\\
\noindent Xi'an Jiaotong-Liverpool University\\
\noindent No.111 Renai Road, Industrial Park, Suzhou, Jiangsu Province, China\\
\noindent e-mail: jinsong.xu@xjtlu.edu.cn

\begin{thebibliography}{100}\setlength{\itemsep}{-1mm}
\bibliographystyle{ieee}
\bibliography{CASSreference}


\bibitem{am} V. Alexeev, S. Mori; Bounding singular surfaces of general type, in \emph{Algebra, Arithmetic and Geometry with Applications}; Springer-Verlag, Berlin Heidelberg, 2004, 143-174.


\bibitem{b1} C. Birkar; Anti-pluricanonical systems on Fano varieties; \texttt{arXiv:1603.05765v3}.

\bibitem{b2} C. Birkar; Singularities of linear systems and boundedness of Fano varieties; \texttt{arXiv:1609.05543}.

\bibitem{bchm} C. Birkar, P. Cascini, C. Hacon, J. McKernan; Existence of minimal models for varieties of log general type; J. Amer. Math. Soc. 23, 2010, 405-468.

\bibitem{br1} M. Brion; Some basic results on actions of nonaffine algebraic groups; in \emph{Symmetry and spaces}, Progress in  Mathematics \textbf{278}, Birkh\"{a}user Boston, MA, 2010.

\bibitem{br2} M. Brion; On connected automorphism groups of algebraic varieties; J. Ramanujan Math. Soc. \textbf{28A} (Special Issue-2013) 41-54.

\bibitem{br3} M. Brion; On extensions of algebraic groups with finite quotient; \texttt{arXiv:1503.06564}.

\bibitem{bz} T. Bandman, Yu. G. Zarhin; Jordan groups, conic bundles and abelian varieties; Algebraic Geometry \textbf{4} (2) (2017) 229-246.

\bibitem{c} S. Cantat; Morphisms between Cremona groups, and characterization of rational varieties; Compositio Math. \textbf{150} (2014) 1107-1124.

\bibitem{g} A. Guld; Boundedness properties of automorphism groups of forms of flag varieties; \texttt{arXiv:1806.05400}.

\bibitem{ghs} T. Graber, J. Harris, J. Starr; Families of rationally connected varieties; J. Amer. Math. Soc. \textbf{16} (2003) no.1, 57-67.

\bibitem{h1} M. Hanamura; On the birational automorphism groups of algebraic varieties; Compositio Math. \textbf{63} (1987) 123-142.

\bibitem{h2} M. Hanamura; Structure of birational automorphism groups, I: non-uniruled varieties; Invent. Math. \textbf{93} (1988) 383-403.


\bibitem{ho} G. P. Hochschild; \emph{Basic Theory of Algebraic Groups and Lie Algebras}; Graduate Texts in Mathematics \textbf{75}. Springer-Verlag, New York-Berlin, 1981.


\bibitem{km} J. Koll\'{a}r, S. Mori; \emph{Birational Geometry of Algebraic Varieties}; Cambridge Tracts in Mathematics, 134, Cambridge University Press 1998.

\bibitem{kmmt} J. Koll\'{a}r, S. Mori, Y. Miyaoka, H.Takagi; Boundedness of canonical $\Q$-Fano threefolds; Proc. Japan. Acad. \textbf{76}, Ser. A (2000), 73-77. 

\bibitem{mz} S. Meng, D.-Q. Zhang; Jordan property for non-linear algebraic groups and projective varieties; \texttt{arXiv:1507.02230}.

\bibitem{m} J. Milne; \emph{Algebraic groups. The theory of group schemes of finite type over a field}; Cambridge Studies in Advanced Mathematics \textbf{170}, Cambridge University Press, Cambridge, 2017.

\bibitem{ps1} Yu. G. Prokhorov, C. Shramov; Jordan property for Cremona groups; Amer. J. Math. \textbf{138} (2016), 403-418.

\bibitem{ps2} Yu. G. Prokhorov, C. Shramov; Jordan property for groups of birational selfmaps; Compositio Math. \textbf{150} (2014) 2054-2072.

\bibitem{ps3} Yu. G. Prokhorov, C. Shramov; $p$-subgroups in the space Cremona group; \texttt{arXiv:1610.02990v4}.

\bibitem{ps4} Yu. G. Prokhorov, C. Shramov; Finite groups of birational selfmaps of threefolds; \texttt{arXiv:1611.00789v2}.


\bibitem{r1} M. Rosenlicht; Some basic theorems of algebraic groups; Amer. J. Math. \textbf{78} (1956), 401-443.

\bibitem{r2} M. Rosenlicht; Another proof of a theorem on rational cross sections; Pacific J. Math. \textbf{20} (1967) 129-133.


\bibitem{s1} J. -P. Serre; Bounds for the orders of the finite subgroups of $G(k)$, in \emph{Group representation theory} (EPFL Press, Lausanne, 2007), 405-450.

\bibitem{s2} J. -P. Serre; A Minkowski style bound for the orders of the finite subgroups of the Cremona group of rank $2$ over an arbitrary field, Modc. Math. J. \textbf{9} (2009), no.1. 193-208.

\bibitem{sv} C. Shramov, V. Vologodsky; Automorphisms of pointless surfaces; \texttt{arXiv:1807.06477}.

\bibitem{sp} T. A. Springer; \emph{Linear Algebraic Groups, Second Edition}; Progress in Mathematics \textbf{9}, Birkh\"auser Boston, Inc., Boston, MA, 1998. 


\bibitem{x} J. Xu; Homogeneous fibrations on log Calabi-Yau varieties; \texttt{arXiv:1611.01804}

\bibitem{z} Yu. G. Zarhin; Theta groups and products of abelian and rational varieties, Proc. Edinb. Math. Soc. (2) \textbf{57} (2014), 299-304.
\end{thebibliography}
\end{document}